\documentclass[12pt, a4paper]{article}
\usepackage{amsthm, amsfonts, amsmath, amssymb, float, graphicx, graphics, epstopdf, subfigure, epigraph, appendix, mathrsfs}
\usepackage[left=3 cm,top=3cm,right=3cm, bottom = 4cm]{geometry}
\usepackage[all]{xy}
\usepackage{hyperref}

\usepackage{color}

\newcommand{\naturals}{\mathbb{N}}

\newcommand{\reals}{\mathbb{R}}
\newcommand{\indicator}{\mathbf{1}}
\newcommand{\pr}{\mathbb{P}}

\newcommand{\ex}{\mathbb{E}}

\newcommand{\D}{\mathscr{K}}
\newcommand{\Lim}{\widehat{X}^{(1)}}

\newcommand{\Mat}{\mathcal{M}}
\newcommand{\SMat}{\mathcal{P}}
\newcommand{\Simp}{\mathcal{S}}
\newcommand{\ed}{\stackrel{d}{=}}

\newcommand{\aas}{\stackrel{a.s.}{\longrightarrow}}

\newcommand{\dt}{\bullet}
\newcommand{\ns}{\mspace{0.0mu}}
\newcommand{\law}{\mathcal{L}}
\newcommand{\I}{[\mathbf{I}]}
\newcommand{\II}{[\mathbf{II}]}
\newcommand{\III}{[\mathbf{III}]}
\newcommand{\Aa}{[\mathbf{A1}]}
\newcommand{\Ab}{[\mathbf{A2}]}
\newcommand{\Ac}{[\mathbf{A3}]}
\newcommand{\ca}{[\mathbf{C1}]}
\newcommand{\cb}{[\mathbf{C2}]}
\newcommand{\vect}[1]{\boldsymbol{#1}}

\providecommand{\abs}[1]{\lvert#1\rvert}
\providecommand{\norm}[1]{\lVert#1\rVert}

\allowdisplaybreaks[1]

\newcounter{lemmacount} 
\newcounter{corcount} 
\newcounter{propcount} 
\newcounter{examples} 
\newcounter{remarks} 

\newtheorem{theorem}{Theorem}
\newtheorem{lemma}[lemmacount]{Lemma}
\newtheorem{proposition}[propcount]{Proposition}
\newtheorem{corollary}[corcount]{Corollary}

\newenvironment{example}[1][]{\medskip\noindent {\textbf{Example
\arabic{examples}.$\ns$} \stepcounter{examples}}}

\newenvironment{remark}[1][]{\medskip\noindent {\textbf{Remark
\arabic{remarks}.$\,$} \stepcounter{remarks}}}

\begin{document}

\title{A characterisation of transient random walks on stochastic matrices with Dirichlet distributed limits}

\author{S.\ McKinlay
\footnote{
Department of Mathematics and Statistics, University of Melbourne, Parkville 3010, Australia. E-mail: shaunam@ms.unimelb.edu.au.
}}

\date{}

\maketitle

\begin{abstract}
\noindent
We characterise the class of distributions of random stochastic matrices $X$ with the property that the products $X(n)X(n-1) \cdots X(1)$ of i.i.d.\ copies $X(k)$ of $X$ converge a.s.\ as $n \rightarrow \infty$ and the limit is Dirichlet distributed. This extends a result by Chamayou and Letac (1994) and is illustrated by several examples that are of interest in applications.\vspace{0.2cm}


\end{abstract}


\section{Introduction}\label{S_Intro}

Let $X$ be a random $d \times d$ matrix with non-negative entries, and $\{X(n)\}_{n \ge 1}$ be an i.i.d.\ sequence of random matrices with the same distribution as $X$. In this paper, we study the limit of the left products
\begin{equation}\label{e_prod}
X(n,1) := X(n) X(n-1) \cdots X(1)
\end{equation}
as $n \rightarrow \infty$ for a certain class of random matrices $X$. Clearly, the distribution of \eqref{e_prod} is equal to that of right product
\begin{equation}\label{e_rprod}
X(1,n) := X(1) X(2) \cdots X(n),
\end{equation}
and therefore any assertions concerning the distribution of \eqref{e_prod} as $n \rightarrow \infty$ apply to that of \eqref{e_rprod} as well.

The asymptotic behaviour of the left products \eqref{e_prod} was apparently first studied by Bellman \cite{Bellman}, who showed, under certain conditions, that $\lim_{n \rightarrow \infty} n^{-1} \ex \log X(n,1)_{i,j}$ exists, where $X(n,1)_{i,j}$ is the entry in the $i$th row and $j$th column of $X(n,1)$. This result was later strengthened to almost sure convergence in \cite{FurstenbergKesten}, where the behaviour of the norms $\norm{X(n,1)} := \max_{i \le d} \sum_{j=1}^d \abs{X(n,1)_{i,j}}$ was studied.

Several equivalent conditions for convergence in distribution of the right products \eqref{e_rprod} without normalisation were given in \cite{KestenSpitzer}. See also \cite{Bougerol} and \cite{Mukherjea} for surveys on the properties of the products \eqref{e_prod} and \eqref{e_rprod}.

We consider the case when the random matrix $X$ is stochastic, i.e., where all the row sums of $X$ equal one. The study of the products \eqref{e_prod} in this case was apparently initiated by Rosenblatt in \cite{Rosenblatt}, who applied existing results for compact semigroups. Recent results on the infinite product of deterministic and random stochastic matrices can be found in \cite{Touri}.

The class of stochastic matrices is a semigroup under matrix multiplication, and therefore the products \eqref{e_prod} and \eqref{e_rprod} generate a left and a right random walk on stochastic matrices by
\begin{equation}\label{lrw}
n \mapsto \; X(n,1)
\end{equation}
and
\begin{equation}\label{rrw}
n \mapsto \:X(1,n),
\end{equation}
respectively.

The right random walk \eqref{rrw} is central to several related problems involving distributed averaging. These include distributed computation, distributed optimization, distributed estimation, and distributed coordination (see \cite{Touri} and references therein). They are also related to certain Markov processes called Potlatch models. These models were introduced in \cite{HolleyLiggett} and \cite{LiggettSpitzer}, and are described in their simplest form in \cite{KestenSpitzer}. We will discuss these models further in Section \ref{SS_Potlatch}. A comprehensive reference for random walks on stochastic matrices and more general semigroups is \cite{Hognas}, where the transient random walk generated by random stochastic matrices satisfying conditions $\I$--$\III$ below is presented in Appendix B.

In \cite{ChamLetac}, Chamayou and Letac study the left products \eqref{e_prod} for random stochastic matrices $X$ satisfying the following conditions:

\begin{itemize}
\item[{$\I$}] The rows of $X$ are independent.
\item[{$\II$}] The rows of $X$ are Dirichlet distributed.
\item[{$\III$}] Letting $(\alpha_{i,1}, \ldots, \alpha_{i,d})$ be the Dirichlet parameters of the $i$th row of $X$, we have $\sum_{j=1}^d\alpha_{i,j} = \sum_{j=1}^d\alpha_{j,i}$ for $i=1, \ldots, d$.
\end{itemize}
It is shown in \cite{ChamLetac} that the above conditions are sufficient to ensure that:

\begin{itemize}
\item[{$\Aa$}] The products $X(n,1)$ converge a.s.\ to some random matrix $\widehat{X}$ as $n \rightarrow \infty$.
\item[{$\Ab$}] The limit $\widehat{X}$ has  identical rows a.s.
\item[{$\Ac$}] The rows of $\widehat{X}$ are Dirichlet distributed.
\end{itemize}
Unfortunately, conditions $\I$--$\III$ are quite restrictive. Surprisingly, it turns out assertions $\Aa$--$\Ac$ remain true under much broader conditions, and that none of the conditions $\I$--$\III$ is necessary. 

Denote by $\D_d$ the class of all distributions of a random $d \times d$ stochastic matrix $X$ such that $\Aa$--$\Ac$ hold. In this note, we extend the result in \cite{ChamLetac} by providing a charaterisation theorem (Theorem \ref{ext2}) for the class $\D_d$. We find that $\D_d$ is much broader than the class of distributions satisfying $\I$--$\III$. In addition, once the Dirichlet parameters of the limiting distribution are known, it is elementary to verify that a given distribution belongs to the class $\D_d$.

Chamayou and Letac's result above relies on a theorem (which we present as Theorem \ref{T1} below) that admits a rather elegant proof suggested by S. Lauritzen and also presented in \cite{ChamLetac}. The authors of \cite{ChamLetac} pointed out, using Erd\"os's remark on the existence of God's book for the best proofs, that S. Lauritzen's proof could be taken from that book. The proof of our Theorem \ref{ext1} which contains Theorem \ref{T1} as a special case, is also rather simple and concise. Similar to S. Lauritzen's proof, it is based on an insightful observation (Lemma \ref{t_indep} below) that, in our case, extends the well-known result by Pitman on scale independent functions of Gamma distributed random variables (see \cite{Pitman}).

The Dirichlet distribution appearing in the limit $\Ac$ is widely used in statistical applications. These include the modelling of compositional data (see e.g.\ \cite{Aitchison}), Bayesian analysis, statistical genetics, and nonparametric inference (see \cite{Ng} for a recent survey on the Dirichlet distribution and its applications).

The paper is structured as follows. Section \ref{S_Results} contains notation and the main results, while their proofs are presented in Section \ref{S_Proofs}. Examples and applications to random exchange models, random nested simplices, and a service network with polling are presented in Section \ref{S_App}.


\section{Notation and Main Results}\label{S_Results}

For positive integers $r$ and $c$, denote by $\Mat_{r,c}$ the set of $r \times c$ matrices $(p_{i,j})_{i=1}^{r} \ns_{j=1}^{c}$ such that all $p_{i,j} \ge 0$, and by $\SMat_{r,c} \subset \Mat_{r,c}$ its subclass of matrices with $\sum_{j=1}^c p_{i,j} = 1$, $i = 1, \ldots, r$. Let  $\Mat^+_{r,c}$ and $\SMat^+_{r,c}$ be the subclass of all positive matrices from $\Mat_{r,c}$ and $\SMat_{r,c}$, respectively (by a positive matrix/vector we mean a matrix/vector with all positive entries), and set $\reals_+ := (0,\infty)$. Clearly, $\SMat_{d} := \SMat_{d,d}$ consists of all $d \times d$ stochastic matrices, with $\SMat^+_{d} := \SMat^+_{d,d}$ its subclass of positive stochastic matrices (we similarly define $\Mat_{d} := \Mat_{d,d}$ and $\Mat^+_{d} := \Mat^+_{d,d}$). For a matrix $(Z_{i,j}) \in \Mat_{r,c}$, we denote its row and column sums by $Z_{i \dt} := \sum_{j=1}^c Z_{i,j}$, $i=1, \ldots, r$, and $Z_{\dt j} := \sum_{i=1}^r Z_{i,j}$, $j=1, \ldots, c$, respectively. Similarly, for a vector $(Y_1, \ldots, Y_c)$, we denote the sum of its components by $Y_{\dt} := \sum_{i=1}^c Y_i$.

For a random element $X$, we denote its distribution by $\law(X)$, and write $X \sim F$ if $F = \law(X)$. If $Y \sim F_Y$ is independent of $X \sim F_X$, we write $F_X \otimes F_Y$ for the law of $(X,Y)$.

We denote by $\Gamma_u$ the Gamma distribution with scale parameter $1$ and shape parameter $u>0$, with density
\begin{equation*}
\frac{x^{u-1} e^{-x}}{\Gamma(u)} \indicator_{(0,\infty)}(x).
\end{equation*}
For a positive vector $\vect{a} = (a_1, \ldots, a_d)$, we denote by $D_{\vect{a}}$ the \emph{Dirichlet distribution} on the simplex $\Simp_d := \SMat_{1,d}$ with density
\begin{equation*}
\Gamma(a_{\dt}) \prod_{i=1}^d \frac{x_i^{a_i-1}}{\Gamma(a_i)}, \quad (x_1, \ldots, x_d) \in \Simp_d,
\end{equation*}
with respect to the ($d-1$)-dimensional volume measure on $\Simp_d$. We denote by $G_{\vect{a}}$ the distribution $\Gamma_{a_1} \otimes \cdots \otimes \Gamma_{a_d}$ on $\Mat_{1,d}$.

For $A = (\alpha_{i,j}) \in \Mat^+_{r,c}$, the \emph{Dirichlet distribution} $D_A$ on $\SMat_{r,c}$ is the law of the matrix $X = (X_{i,j}) \in \SMat_{r,c}$, such that $X^{(i)} :=(X_{i,1}, \ldots, X_{i,c}) \sim D_{(\alpha_{i,1}, \ldots, \alpha_{i,c})}$ and $X^{(1)}, \ldots, X^{(r)}$ are independent. Similarly, by $G_A$ we denote the law of the matrix $Z = (Z_{i,j}) \in \Mat_{r,c}$, such that $Z^{(i)} :=(Z_{i,1}, \ldots, Z_{i,c}) \sim G_{(\alpha_{i,1}, \ldots, \alpha_{i,c})}$ and $Z^{(1)}, \ldots, Z^{(r)}$ are independent.

Let $A = (\alpha_{i,j}) \in \Mat^+_{r,c}$. The following theorems are the main results in \cite{ChamLetac}. They are extensions of earlier theorems by Van Assche \cite{VanAssche}, who proved them in the case $c=r=2$ and all $\alpha_{i,j} = p>0$ (see also \cite{Volodin} for an extension to finite dimensions of the result in \cite{VanAssche})

\begin{theorem}\label{T1}
\emph{(\cite{ChamLetac})} If $(\vect{Y},X) \sim D_{(\alpha_{1 \dt}, \ldots, \alpha_{r \dt})} \otimes D_A$, then $\vect{Y}X \sim D_{(\alpha_{\dt 1}, \ldots, \alpha_{\dt c})}$. 
\end{theorem}

\begin{theorem}\label{T2}
\emph{(\cite{ChamLetac})} If $r=c=d$, $X \sim D_A$, and 
\begin{equation}\label{condA}
(\alpha_{1 \dt}, \ldots, \alpha_{d \dt}) = (\alpha_{\dt 1}, \ldots, \alpha_{\dt d}),
\end{equation}
then $\law(X) \in \D_d$, and $\Lim \sim D_{(\alpha_{1 \dt}, \ldots, \alpha_{d \dt})}$. Furthermore, if $\vect{Y}$ is a random vector taking values in $\Simp_d$ that is independent of $X$, then $\vect{Y}X \ed \vect{Y}$ iff $\vect{Y} \sim D_{(\alpha_{1 \dt}, \ldots, \alpha_{d \dt})}$.
\end{theorem}

Theorem \ref{T2} is actually a simple consequence of Theorem \ref{T1}. Likewise, our extension of Theorem \ref{T2} (Theorem \ref{ext2} below) is based on the following result that, as we show at the end of this section, implies Theorem \ref{T1}. 

\begin{theorem}\label{ext1}
Let  $\vect{t} = (t_1, \ldots, t_r) \in \reals_+^r$ and $\vect{s} = (s_1, \ldots, s_c) \in \reals_+^c$ with $t_{\dt} = s_{\dt}$. Suppose $X$ is a random element of $\Mat_{r,c}$ independent of both $\vect{Y} \sim D_{\vect{t}}$ and $\vect{V} \sim G_{\vect{t}}$. Then
\begin{equation}\label{e_ext1}
\vect{Y}X \sim D_{\vect{s}} \;\; \text{ iff } \;\; \vect{V}X \sim G_{\vect{s}}. 
\end{equation}
\end{theorem}

To state the next theorem, we will need the following conditions on the random matrix $X \in \SMat_d$.\vspace{0.2cm}\\
$\ca$ There exists a $\vect{t} \in \reals_+^d$ such that, for a random vector $\vect{V} \sim G_{\vect{t}}$ independent of $X$, one has $\vect{V}X \ed \vect{V}$.\vspace{0.2cm}\\
$\cb$ For an i.i.d.\ sequence $\{X(n)\}_{n \ge 1}$ with $X(1) \ed X$, there exists an $m < \infty$ such that $\pr(X(m,1) \in \SMat^+_d) > 0$.

\begin{remark}
Note that, unlike the conditions of Theorem \ref{T2} (under which $\ca$ holds), condition $\ca$ does not mean that $X$ must be positive. In Section \ref{S_App}, we provide examples where $\ca$ is met, but most of the entries of $X$ are zeros.
\end{remark}

\begin{theorem}\label{ext2}
\begin{itemize}
\item[\emph{(i)}] Relation
\begin{equation}\label{e_converge}
\law(X) \in \D_d
\end{equation}
holds iff $\ca$ and $\cb$ are met for $X$.

\item[\emph{(ii)}] If \eqref{e_converge} holds, then $\Lim \sim D_{\vect{t}}$, where the vector $\vect{t}$ is the same as in  $\ca$, and if $\vect{Y}$ is a random element of $\Simp_d$ independent of $X$, then $\vect{Y}X \ed \vect{Y}$ iff $\vect{Y} \ed \Lim$.
\end{itemize}
\end{theorem}

We conclude this section by showing that the assertions of Theorems \ref{T1} and \ref{T2} do follow from those of our Theorems \ref{ext1} and \ref{ext2}. Let $\vect{\xi} = (\xi_1, \ldots, \xi_d) \sim G_{\vect{t}}$ for some $\vect{t} \in \reals_+^d$. Then it is well-known that (see for instance formula 2.1.2 in \cite{Ng})
\begin{equation}\label{e_DirGam}
\left( \frac{\xi_1}{\xi_{\dt}}, \ldots, \frac{\xi_d}{\xi_{\dt}} \right) \sim D_{\vect{t}}.
\end{equation}

A remarkable characteristic property of the gamma distribution is as follows (\cite{Lukacs}). Suppose that $\eta_1, \eta_2 > 0$ are independent non-degenerate random variables. Then $\eta_1$  and $\eta_2$ are gamma distributed with a common scale parameter iff $\eta_1+\eta_2$ is independent of $\eta_1/(\eta_1+\eta_2)$. It follows that 
\begin{equation}\label{gamprop}
\xi_{\dt} \quad \text{is independent of} \quad \left( \frac{\xi_1}{\xi_{\dt}}, \ldots, \frac{\xi_d}{\xi_{\dt}} \right),
\end{equation}
and so
\begin{equation}\label{gamprop2}
(\xi_1, \ldots, \xi_d) \ed \left(\frac{\xi_1}{\xi_{\dt}}, \ldots, \frac{\xi_d}{\xi_{\dt}} \right)\tilde{\xi_{\dt}},
\end{equation}
where $(\tilde{\xi}_1, \ldots, \tilde{\xi}_d)$ is an independent copy of $\vect{\xi}$.

We will now use \eqref{e_DirGam} and relation \eqref{gamprop2} to show that Theorem \ref{T1} follows from Theorem \ref{ext1}. Let $A = (\alpha_{i,j}) \in \Mat^+_{r,c}$ and $(\vect{V},Z) \sim G_{(\alpha_{1 \dt}, \ldots, \alpha_{r \dt})} \otimes G_A$. Then, we have from  \eqref{e_DirGam} that
\begin{equation}\label{ZrepX}
X :=
\begin{pmatrix}
\frac{Z_{1,1}}{Z_{1\dt}} & \dotsb & \frac{Z_{1,c}}{Z_{1\dt}}\\
\vdots & \ddots & \vdots\\
\frac{Z_{r,1}}{Z_{r \dt}} & \dotsb & \frac{Z_{r,c}}{Z_{r \dt}}\\
\end{pmatrix} \sim D_A,
\end{equation}
and
\begin{equation}\label{Ydist}
\vect{Y} := \left( \frac{V_1}{V_{\dt}}, \cdots, \frac{V_r}{V_{\dt}} \right) \sim D_{(\alpha_{1 \dt}, \ldots, \alpha_{r \dt})}.
\end{equation}
Therefore
\begin{equation}\label{VXdist}
(\vect{V},X) \sim G_{(\alpha_{1 \dt}, \ldots, \alpha_{r \dt})} \otimes D_A,
\end{equation}
and
\begin{equation}\label{YXdist}
(\vect{Y},X) \sim D_{(\alpha_{1 \dt}, \ldots, \alpha_{r \dt})} \otimes D_A.
\end{equation}
Now
\begin{equation}\label{e_dirmatrix}
\vect{V}X = 
\sum_{k=1}^r \left( \frac{Z_{k,1}}{Z_{k \dt}}, \ldots, \frac{Z_{k,c}}{Z_{k \dt}} \right) V_k,
\end{equation}
and relations \eqref{gamprop2} and \eqref{VXdist} imply that the random vector on the right hand side of \eqref{e_dirmatrix} is distributed as
\begin{equation}\label{e_dirmatrix2}
\sum_{k=1}^r (Z_{k,1}, \ldots, Z_{k,c}) = (Z_{\dt 1}, \ldots, Z_{\dt c}) \sim G_{(\alpha_{\dt 1}, \ldots, \alpha_{\dt c})}.
\end{equation}
It follows from Theorem \ref{ext1} that for $\vect{Y}$ satisfying \eqref{YXdist}, one has $\vect{Y}X \sim D_{(\alpha_{\dt 1}, \ldots, \alpha_{\dt c})}$, thus establishing the claim of Theorem \ref{T1}.

To see that Theorem \ref{T2} follows from Theorems \ref{ext1} and \ref{ext2}, we suppose that $\{X(n)\}_{n \ge 1}$ are i.i.d.\ with law $D_A$, and $A$ satisfies \eqref{condA}. By Theorem \ref{T1}, if $\vect{Y} \sim D_{(\alpha_{1\dt}, \ldots, \alpha_{d \dt})}$ is independent of $X(1)$, then $\vect{Y}X(1) \sim D_{(\alpha_{\dt 1}, \ldots, \alpha_{\dt d})}$, and Theorem \ref{ext1} implies that $\vect{V}X(1) \ed \vect{V}$ for $\vect{V} \sim G_{(\alpha_{\dt 1}, \ldots, \alpha_{\dt d})}$ independent of $X(1)$. Thus $\ca$ holds with $\vect{t} = (\alpha_{\dt 1}, \ldots, \alpha_{\dt d})$. Since $\pr(X(1) \in \SMat^+_d) = 1$, $\cb$ is satisfied as well, and the assertion of Theorem \ref{T2} follows by applying Theorem \ref{ext2}.

 
\section{Proofs}\label{S_Proofs}

A remarkable observation made by Pitman \cite{Pitman} is the following extension of \eqref{gamprop}. Let, as above, $\vect{\xi} \sim G_{\vect{t}}$, and $f: \reals^d \rightarrow \reals$ be a \emph{scale independent} measurable function, i.e., for any $a \neq 0$,
\begin{equation*}
f(ax_1, \ldots, ax_d) \equiv f(x_1, \ldots, x_d).
\end{equation*}
Then the random variable $f(\xi_1, \ldots, \xi_d)$ is independent of $\xi_{\dt}$.

The proof of Theorem \ref{ext1} is based on the following extension of that observation to random functions.

\begin{lemma}\label{t_indep}
Let $(\Omega, \mathcal{F}, \pr)$ be a probability space, $(E, \mathcal{E})$ a measurable space, and $X: \Omega \rightarrow E$ a random element. Suppose $H: \reals^r \times E \rightarrow \reals_+$ is jointly measurable and, for any $a \neq 0$ and $\omega \in \Omega$,
\begin{equation*}
H(ay_1, \ldots, ay_r, X(\omega)) = H(y_1, \ldots, y_r, X(\omega))
\end{equation*}
for all $(y_1, \ldots, y_r) \in \reals^r$. If $\vect{V} = (V_1, \ldots, V_r) \sim G_{\vect{t}}$, $\vect{t} \in \reals_+^r$, is independent of $X$, then $V_{\dt}$ is independent of $H(\vect{V}, X)$.
\end{lemma}

\begin{proof}
Let $\phi(s,u)$, $(s,u) \in \reals_+^2$, denote the joint Laplace transform of $V_{\dt}$ and $H(\vect{V},X)$. Then
\begin{align*}
\phi(s,u) &= \ex e^{-s(V_1 + \cdots + V_r) - uH(V_1, \ldots, V_r,X)}\\
&= C_{\vect{t}} \int_0^{\infty} \cdots \int_0^{\infty} v_1^{t_1-1} \cdots v_r^{t_r-1} e^{-(1+s)\sum v_k} \ex e^{-uH(v_1, \ldots, v_r, X)} dv_1 \cdots dv_r,
\end{align*}
where $C_{\vect{t}}^{-1} = \prod_{k=1}^r \Gamma(t_k)$. Making the substitution $y_k = (1+s)v_k$ and observing that $H(v_1, \ldots, v_r, X) \ed H(y_1, \ldots, y_r, X)$ by virtue of scale independence, we have that
\begin{equation*}
\phi(s,u) = \frac{C_{\vect{t}}}{(1+s)^{\sum t_k}} \int_0^{\infty} \cdots \int_0^{\infty}  y_1^{t_1-1} \cdots y_r^{t_r-1} e^{-\sum y_k} \ex e^{-uH(y_1, \ldots, y_r, X)} dy_1 \cdots dy_r,
\end{equation*}
which is the product of two functions, one depending on $s$, and the other on $u$. Therefore $V_{\dt}$ and $H(\vect{V},X)$ are independent as claimed.
\end{proof}

The next result is an obvious consequence of Lemma \ref{t_indep}.

\begin{corollary}\label{c_indep}
Let $\vect{V} = (V_1, \ldots, V_r) \sim G_{\vect{t}}$, and $X = (X_{i,j}) \in \Mat_{r,c}$ be random elements independent of each other, $X$ having positive row sums a.s. Define the function $H(v_1, \ldots, v_r, X) = (H_1(v_1, \ldots, v_r, X), \ldots,$ $ H_c(v_1, \ldots, v_r, X))$ by
\begin{equation}\label{e_hdef}
H_j(v_1, \ldots, v_r, X) := \sum_{i=1}^r \frac{v_iX_{i,j}}{v_{\dt}}, \quad 1 \le j \le c.
\end{equation}
Then $H(\vect{V}, X)$ is independent of $V_{\dt}$.
\end{corollary}

\begin{proof}[Proof of Theorem \ref{ext1}.] Suppose the right relation in \eqref{e_ext1} holds, i.e., $\vect{V}X \sim G_{\vect{s}}$ for $\vect{V} \sim G_{\vect{t}}$ independent of $X$. Then, by Corollary \ref{c_indep}, for the scale independent function $H$ defined in \eqref{e_hdef}, the random vector $H(V_1, \ldots, V_r, X) \equiv \left(\frac{V_1}{V_{\dt}}, \ldots, \frac{V_r}{V_{\dt}} \right)X$ is independent of $V_{\dt}$, and therefore
\begin{equation}\label{VX}
\vect{V}X = \left(\frac{V_1}{V_{\dt}}, \ldots, \frac{V_r}{V_{\dt}} \right)X V_{\dt} \ed \left(\frac{V_1}{V_{\dt}}, \ldots, \frac{V_r}{V_{\dt}} \right)X \widetilde{V}_{\dt},
\end{equation}
where $(\widetilde{V}_1, \ldots,\widetilde{V}_r) \sim G_{\vect{t}}$ is independent of $(\vect{V},X)$. Since $\vect{V}X \sim G_{\vect{s}}$, for $\vect{Z} := (Z_1, \ldots, Z_c) \sim G_{\vect{s}}$, one has
\begin{equation}\label{VXZ}
\vect{V}X \ed \vect{Z} \ed \left(\frac{Z_1}{Z_{\dt}}, \ldots, \frac{Z_c}{Z_{\dt}} \right)\widetilde{Z}_{\dt},
\end{equation}
$(\widetilde{Z}_1, \ldots,\widetilde{Z}_c)$ being an independent copy of $\vect{Z}$ (cf.~\eqref{gamprop2}).

Equating the logarithms of the components of the vectors on the right hand sides of \eqref{VX} and \eqref{VXZ}, we obtain that
\begin{align}
&\left( \ln \left( \frac{\sum_{i=1}^r V_i X_{i,1}}{V_{\dt}} \right), \ldots, \ln \left(\frac{\sum_{i=1}^r V_i X_{i,c}}{V_{\dt}} \right) \right) +  \ln(\widetilde{V}_{\dt}) \vect{e}_c\notag\\
&\ed \left( \ln \left( \frac{Z_1}{Z_{\dt}} \right), \ldots,  \ln \left( \frac{Z_c}{Z_{\dt}} \right)\right) + \ln(\widetilde{Z}_{\dt}) \vect{e}_c,\label{e_logVXZ}
\end{align}
where $\vect{e}_c := (1,\ldots, 1) \in \reals^c$.

Since $t_{\dt} = s_{\dt}$, one has $\widetilde{V}_{\dt} \ed \widetilde{Z}_{\dt}$ as both follow $\Gamma_{t_{\dt}} = \Gamma_{s_{\dt}}$, and so letting $\psi$, $\varphi$ and $\chi$ denote the characteristic functions of the first, second (and fourth), and third terms in \eqref{e_logVXZ}, respectively, we have
\begin{equation*}
\psi(u_1, \ldots, u_c) \varphi(u_1, \ldots, u_c) = \chi(u_1, \ldots, u_c) \varphi(u_1, \ldots, u_c).
\end{equation*}
Noting that
\begin{align*}
\varphi(u_1, \ldots, u_c) &= \ex e^{i u_{\dt} \ln(V_{\dt})} = \frac{1}{\Gamma(t_{\dt})} \int_0^{\infty} e^{iu_{\dt} \ln x} x^{t_{\dt}-1} e^{-x} \, dx\\
&= \frac{1}{\Gamma(t_{\dt})} \int_0^{\infty} x^{iu_{\dt}+t_{\dt}-1} e^{-x} \, dx = \frac{\Gamma(t_{\dt}+iu)}{\Gamma(t_{\dt})} \neq 0
\end{align*}
for $t_{\dt}>0$, we conclude that $\psi \equiv \chi$ and therefore the respective random vectors have a common distribution. Hence one has
\begin{equation}\label{e_VXZ}
\left(\frac{V_1}{V_{\dt}}, \ldots,  \frac{V_r}{V_{\dt}} \right)X = \left( \frac{\sum_{i=1}^r V_i X_{i,1}}{V_{\dt}}, \ldots, \frac{\sum_{i=1}^r V_i X_{i,c}}{V_{\dt}} \right) \ed \left(\frac{Z_1}{Z_{\dt}}, \ldots,  \frac{Z_c}{Z_{\dt}} \right) \sim D_{\vect{s}}
\end{equation}
from \eqref{e_DirGam}. Observing that the left hand side of \eqref{e_VXZ} has the form $\vect{Y}X$ for $\vect{Y} \sim D_{\vect{t}}$ independent of $X$ (cf.~\eqref{e_DirGam}), we have that $\vect{Y}X \sim D_{\vect{s}}$, and so the left relation in \eqref{e_ext1} holds.

Conversely, suppose that $\vect{Y}X \sim D_{\vect{s}}$ for $\vect{Y} \sim D_{\vect{t}}$ independent of $X$. Using the same steps as above, but following them in the reverse order, one can easily conclude that $\vect{V}X \sim G_{\vect{s}}$ for $\vect{V} \sim G_{\vect{t}}$ independent of $X$. Theorem \ref{ext1} is proved.
\end{proof}

We will need a simple extension of Proposition~$2.2$ from~\cite{ChamLetac}.

\begin{proposition}\label{prop}
Let $X(1)$ be a random element of $\SMat_d$ satisfying $\cb$, and $\{X(n)\}_{n \ge 1}$ be an i.i.d.\ sequence. Then there exists a random element $\vect{W}$~of $\Simp_d$ such that
\begin{equation}\label{e1_prop}
X(n,1) \aas  \vect{e}_d^T \vect{W}
\end{equation}
as $n \rightarrow \infty$, where $\vect{e}_d^T$ denotes the transpose of $\vect{e}_d$. Furthermore, if $\vect{Y}$ is a random element of $\Simp_d$ that is independent of $X(1)$, then $\vect{Y}X(1) \ed \vect{Y}$ iff $\vect{Y} \ed \vect{W}$. 
\end{proposition}

\begin{proof}
By Proposition $2.2$ in \cite{ChamLetac}, there exists a random element $\vect{W}$ of $\Simp_d$ such that
\begin{equation}\label{e2_prop}
X(nm,1) \aas \vect{e}_d^T \vect{W}
\end{equation}
as $n \rightarrow \infty$, where $m< \infty$ is from $\cb$. For $k \in (nm, (n+1)m) \cap \naturals$, one has
\begin{align*}
X(k,1) &= X(k,nm+1) X(nm,1)\\
&=X(k,nm+1)\vect{e}_d^T \vect{W} +X(k,nm+1) \left(X(nm,1) - \vect{e}_d^T \vect{W} \right)\\
&= \vect{e}_d^T \vect{W} +  o(1) \quad a.s.
\end{align*}
as $n \rightarrow \infty$, since $X\vect{e}_d^T = \vect{e}_d^T$ for any $X \in \SMat_d$. Hence, the limit in \eqref{e1_prop} exists and coincides with that in \eqref{e2_prop}. The second part follows immediately from the proof of Proposition $2.2$ in \cite{ChamLetac}.
\end{proof}

\begin{proof}[Proof of Theorem \ref{ext2}.] (i) Suppose $\ca$ and $\cb$ hold. Then, by Proposition \ref{prop}, there exists a random element $\vect{W}$ of $\Simp_d$ such that
\begin{equation*}
X(n,1) \aas \widehat{X} \equiv \vect{e}_d^T \vect{W}
\end{equation*}
as $n \rightarrow \infty$, and therefore $\Aa$ and $\Ab$ hold. From condition $\ca$ and Theorem \ref{ext1}, it follows that $\vect{Y} X(1) \ed \vect{Y}$ for $\vect{Y} \sim D_{\vect{t}}$ independent of $X(1)$. By Proposition \ref{prop}, the law of $\vect{Y}$ coincides with the distribution of the limit $\widehat{X}^{(i)}, i=1, 2, \ldots, d$, i.e.
\begin{equation}\label{e_limit}
\Lim =  \widehat{X}^{(2)} = \cdots = \widehat{X}^{(d)}\ed \vect{Y} \sim D_{\vect{t}},
\end{equation}
and so $\Ac$ holds. We have proved that $\law(X) \in \D_d$.

Conversely, let $\law(X) \in \D_d$. As $X(m,1)$ converges a.s.\ to $\widehat{X} \in \SMat_d^+$, one must have $X(m,1) \in \SMat_d^+$ for all large enough $m$, and so $\cb$ holds. Furthermore, clearly
\begin{equation*}
\widehat{X} = \lim_{n \rightarrow \infty} X(n,1) = \lim_{n \rightarrow \infty} X(n,2)X(1) = \widehat{X}'X(1),
\end{equation*}
where $\widehat{X}' \ed \widehat{X}$ is independent of $X(1)$. In particular, one has $\widehat{X}^{(1)} X(1) \ed \widehat{X}^{(1)} \sim D_{\vect{t}}$ for some $\vect{t} \in \reals_+^d$. Applying Theorem \ref{ext1}, we have that $\vect{V}X(1) \ed \vect{V}$ for $\vect{V} \sim G_{\vect{t}}$, and so $\ca$ holds.

(ii) If \eqref{e_converge} holds, then by (i) we have that $\ca$ and $\cb$ are satisfied, in which case we have \eqref{e_limit}, where $\vect{t}$ is the same as in $\ca$. Furthermore, for $\vect{Y}$ a random element of $\Simp_d$ that is independent of $X$, one has by Proposition \ref{prop} that $\vect{Y}X \ed \vect{Y}$ iff $\vect{Y} \ed \Lim$, where $\Lim \sim D_{\vect{t}}$ by \eqref{e_limit}.

\end{proof}


\section{Examples and Applications}\label{S_App}

\begin{subsection}{Random exchange models}\label{SS_Potlatch}

Consider the following random exchange model, which is a discrete time analogue of certain continuous time Markov processes, called Potlatch models (see for instance \cite{HolleyLiggett} and \cite{LiggettSpitzer}). Suppose we have $d < \infty$ bins labelled by numbers $1, 2, \ldots, d$, that hold amounts $q_k(n)$, $k=1, \ldots, d$, of a homogeneous commodity at times $n=0,1, 2, \ldots,$ respectively. The dynamics of the model is as follows: at time $n \ge 1$, the vector $\vect{q}(n-1) := (q_1(n-1), \ldots, q_{d}(n-1))$ changes to $\vect{q}(n) := \vect{q}(n-1)X(n)$, where $\{X(n)\}_{n \ge 1}$ are i.i.d.\ random elements of $\SMat_d$ with distribution $\law(X)$, which generates the Markov chain
\begin{equation}\label{e_Potlatch}
\vect{q}(n) = \vect{q}(0) X(1,n), \quad n \ge 1.
\end{equation}
Clearly $\sum_{k=1}^d \vect{q}_k(n) = \sum_{k=1}^d \vect{q}_k(0)$ for all $n$, and without loss of generality, we assume that this quantity is equal to one.

As noted in \cite{ChamLetac}, it is easy to prove that Markov chain \eqref{e_Potlatch} has a stationary distribution when $\law(X) \in \D_d$. Indeed, defining the random maps $K_n: \Simp_d \rightarrow \Simp_d$ by $K_n(\vect{y}) := \vect{y}X(n)$, one has that, for any $\vect{y} \in \Simp_d$,
\begin{equation*}
K_1 \circ \cdots \circ K_n (\vect{y}) = \vect{y}X(n,1) \overset{a.s.}{\longrightarrow} \Lim \; \text{as } n \rightarrow \infty.
\end{equation*}
Applying Proposition $1$ in \cite{ChamLetac2}, we obtain that $\law(\Lim)$ is the stationary distribution of Markov chain \eqref{e_Potlatch}.

The random exchange model \eqref{e_Potlatch} is a higher dimension analogue of the stochastic give-and-take model studied in \cite{DeGroot} and introduced in its deterministic form in the context of human genetics in \cite{Li}. This model is applicable to the study of dynamical systems (see for instance \cite{Bruneau}), e.g.\ it provides a model of a closed economy, where agents exchange real wealth at each step in a random manner.

Now we consider two special cases of random stochastic matrices $X$ with $\law(X) \in \D_d$. The purpose of these examples is to demonstrate that none of the conditions $\Aa$--$\Ac$ from \cite{ChamLetac} is necessary. We will also discuss the respective special cases of random exchange model \eqref{e_Potlatch}. Our first example leads to a generalisation of the result in \cite{ChamLetac} to the case of ``extended" Dirichlet distributions, where the parameters of the distribution are permitted to be zero. Our second example demonstrates that the rows of $X$ need not be independent.

\begin{example}\label{E_EDir}
For a vector $\vect{a} \in \Mat_{1,d}$, we set $D_{\vect{a}}$ to be the weak limit of the distribution $D_{\vect{a} + \epsilon}$ as $\epsilon \downarrow 0$, with the usual interpretation of the sum $\vect{a} + \epsilon$ with $\vect{a} \in \reals^d$ and $\epsilon \in \reals$. In other words, the components of $\vect{Y} \sim D_{\vect{a}}$ that correspond to zero components of $\vect{a}$ are identically zero, whereas the subvector of $\vect{Y}$ consisting of the components $Y_j$ of that random vector that correspond to $a_j>0$ form a usual Dirichlet distributed vector. Likewise, for a matrix $A \in \Mat_{r,c}$, we set $D_A$ to be the weak limit of the distribution of $D_{A+\epsilon}$ as $\epsilon \downarrow 0$. We define the distributions $G_{\vect{a}}$ and $G_A$ in a similar way.

Let $A = (\alpha_{i,j}) \in \Mat_d$ with $\alpha_{i \dt} = \alpha_{\dt i} >0$, for $i=1, \ldots, d$. If $X \sim D_A$ and $X$ satisfies $\cb$, then it turns out that $\law(X) \in \D_d$ and $\Lim  \sim D_{(\alpha_{1 \dt}, \ldots, \alpha_{d \dt})}$, and so $D_{(\alpha_{1 \dt}, \ldots, \alpha_{d \dt})}$ is the stationary distribution of Markov chain \eqref{e_Potlatch}.

Indeed, by Theorem \ref{ext2} it suffices to show that $\vect{V}X \ed \vect{V}$ for $\vect{V} \sim G_{(\alpha_{1 \dt}, \ldots, \alpha_{d \dt})}$ independent of $X$. To do that, let $(\vect{V},Z)  \sim G_{(\alpha_{1 \dt}, \ldots, \alpha_{d \dt})} \otimes G_A$, and define $X \in \SMat_d$ by \eqref{ZrepX}. Then $\vect{V}$ is independent of $X$, and $\vect{V}X \sim G_{(\alpha_{1 \dt}, \ldots, \alpha_{d \dt})}$ by \eqref{e_dirmatrix} and \eqref{e_dirmatrix2}, as required.

In particular, we have obtained the stationary distribution for the following simple model, which, to the best of the author's knowledge, has not previously been studied: at time $n \ge 1$, a uniform proportion of the commodity previously held in bin $k$, $k=1,2, \ldots, d$, is shifted to the (neighbouring) bin $k+1$ (mod $d$). In this case vector $\vect{q}(n)$ is defined by \eqref{e_Potlatch} with 
\begin{equation}\label{e_example1}
X :=
\begin{pmatrix}
U_1 & 1-U_1 & 0 & \cdots & 0\\
0 & U_2 & 1-U_2 & \ddots & \vdots\\
\vdots & \ddots & \ddots & \ddots & 0\\
0 & \cdots & 0 & U_{d-1} & 1-U_{d-1}\\
1-U_d & 0 & \cdots & 0 & U_d\\
\end{pmatrix},
\end{equation}
where the $U_k$, $k=1, \ldots, d$, are i.i.d.\ uniformly distributed random variables on $(0,1)$. Observing that $X$ defined by \eqref{e_example1} satisfies $\cb$ for $m=d-1$, we have from above that $\Lim \sim D_{(2, \ldots, 2)}$, and so $D_{(2, \ldots, 2)}$ is the stationary distribution of Markov chain \eqref{e_Potlatch} with $X(n)$ distributed as \eqref{e_example1}.
\end{example}

\begin{example}
In this example, we consider a random stochastic matrix $X$ with all rows dependent, which generates Markov chain \eqref{e_Potlatch} with the same stationary distribution as that corresponding to the random stochastic matrix \eqref{e_example1}. The behaviour of this model is controlled by the decisions of a ``leader" as follows. At time $n\ge1$, the ``leader" shifts a uniform proportion of the commodity held in bin $1$ to bin $2$. If the proportion shifted is greater than $1/2$, then no other shifts occur in the system at time $n$. However, if the proportion shifted is less than or equal to $1/2$, then the commodity previously held in bin $k$, $k=2,3, \ldots, d$, $d \ge 2$, is shifted to the (neighbouring) bin $k+1$ (mod $d$). In this model, the random vector $\vect{q}(n)$ is given by \eqref{e_Potlatch}, with the random element $X \in \SMat_d$ defined by 
\begin{equation}\label{e_example2}
X :=
\begin{pmatrix}
U & 1-U & 0 & \cdots & 0\\
0 & I & 1-I & \ddots & \vdots\\
\vdots & \ddots & \ddots & \ddots & 0\\
0 & \cdots & 0 & I & 1-I\\
1-I & 0 & \cdots & 0 & I\\
\end{pmatrix},
\end{equation}
where $U$ is a uniform random variable on $(0,1)$, and $I := \indicator_{\{U>1/2\}}$, $\indicator_A$ being the indicator function for event $A$.

We will show that $\law(X) \in \D_d$ and $\widehat{X}^{(1)} \sim D_{(2, \ldots, 2)}$. By Theorem \ref{ext2}, it suffices to show that $\ca$ and $\cb$ hold for $X$, $\vect{t} = (2, \ldots, 2)$ being the vector from $\ca$. It is not hard to directly verify that $X$ defined by \eqref{e_example2}  satisfies $\cb$ for $m=2d-2$. Now let $\vect{V} \sim G_{(2, \ldots, 2)}$ be independent of $X$, and denote by $\varphi(u_1, \ldots, u_d)$ the characteristic function of $\vect{V}X$. Then, setting $u_{d+1} := u_1$, we have
\begin{align*}
\varphi(u_1, \ldots, u_d) &= \ex \exp \left \{i(u_1 U + u_2(1-U)) V_1 + i\sum_{j=2}^d (u_j I + u_{j+1}(1-I))V_j \right\}\\
&= \ex \exp \left\{i\sum_{j=2}^d u_{j+1}V_j \right\} \int_0^{1/2} \frac{ds}{[1-i(u_1s+u_2(1-s))]^2}\\
& \quad + \ex \exp \left\{i \sum_{j=2}^d u_j V_j \right\} \int_{1/2}^{1} \frac{ds}{[1-i(u_1s+u_2(1-s))]^2}\\
&= \Biggr( \frac{(1-iu_2)^2}{\prod_{k=1}^d (1-iu_k)^2} \Biggr) \Biggr( \frac{i}{(1-iu_2)(2i+u_1+u_2)} \Biggr)\\
& \quad + \Biggr( \frac{(1-iu_1)^2}{\prod_{k=1}^d (1-iu_k)^2} \Biggr) \Biggr( \frac{i}{(1-iu_1)(2i+u_1+u_2)} \Biggr)\\
&=  \prod_{k=1}^d (1-iu_k)^{-2},
\end{align*}
which is the characteristic function of $\vect{V}$ as well, and so we conclude that $\vect{V}X \ed \vect{V}$. Therefore $\ca$ holds with $\vect{t} = (2, \ldots, 2)$, as required.
\end{example}

\end{subsection}

 
 \begin{subsection}{Random nested simplices}\label{S_RNT}
The study of random triangles and, in particular, nested sequences of random triangles, has been extensive in the probabilistic literature (see for instance \cite{LetacScarsini} and references therein). In one such study \cite{LetacScarsini}, the authors use Theorem $1.2$ in \cite{ChamLetac} (see Theorem \ref{T2} above) to prove that the limiting distribution of certain random nested simplices is Dirichlet. More precisely, they consider a real $(d-1)$-dimensional affine space $E$ and a convex hull of $d$ points $\vect{p}_1, \ldots, \vect{p}_d$ in $E$ which forms a simplex $S$. Choosing $d$ points $\vect{p}_1(1), \ldots, \vect{p}_d(1)$ in $S$ and taking their convex hull yields a new simplex $S(1) \subset S$. Similarly, choosing points $\vect{p}_1(2), \ldots, \vect{p}_d(2)$ in $S(1)$ to obtain $S(2)$ spanned by them, and continuing in this fashion, we obtain a sequence $\{S(n)\}_{n \ge 1}$ of nested simplices.

Suppose that, in the above procedure, the points are chosen at random using the following mechanism. Let $\{X(n)\}_{n \ge 1}$ be a sequence of i.i.d.\ elements of $\SMat_d$ with distribution $D_A$, $A \in \Mat_d$ satisfying \eqref{condA}. Then, for each $n=1,2,\ldots$ and $k=1,\ldots,d$, the point $\vect{p}_k(n)$ is specified by its barycentric coordinates given by $X^{(k)}(n)$ with respect to $(\vect{p}_1(n-1), \ldots, \vect{p}_d(n-1))$. As shown in \cite{LetacScarsini}, the so defined simplices $S(n)$ shrink to a random point $\vect{Y} \in E$ as $n \rightarrow \infty$:
\begin{equation}\label{e_tet}
\vect{p}_k(n) \aas \vect{Y}, \quad k=1, \ldots, d,
\end{equation}
where the barycentric coordinates of $\vect{Y}$ with respect to $(\vect{p}_1, \ldots, \vect{p}_d)$ are Dirichlet distributed with parameter $(\alpha_{1 \dt}, \ldots, \alpha_{d \dt})$. We extend this result as follows (see Section $2$ in \cite{LetacScarsini} for definitions related to affine spaces, affine frames and barycentric coordinates):

\begin{theorem}\label{T_RNT}
If $X(1)$ satisfies $\ca$ and $\cb$, then \eqref{e_tet} holds true. Furthermore, the barycentric coordinates of $\vect{Y}$ with respect to $(\vect{p}_1, \ldots, \vect{p}_d)$ have distribution $D_{\vect{t}}$, $\vect{t}$ being the vector from $\cb$.
\end{theorem}

\begin{proof}
The proof of Theorem \ref{T_RNT} basically repeats that of Theorem $2.2$ from \cite{LetacScarsini}, except we use Theorem \ref{ext2} in place of Theorem $\ref{T2}$ above. For $\vect{B}(0) \in \Simp_d$, define $\vect{Z}(n) \in S(n)$ as the point with barycentric coordinates $\vect{B}(0)$ with respect to the affine frame $(\vect{p}_1(n), \ldots, \vect{p}_d(n))$, and denote by $\vect{B}(n)$ the barycentric coordinates of $\vect{Z}(n)$ with respect to the affine frame $(\vect{p}_1, \ldots, \vect{p}_d)$. Then, as shown in the proof of Theorem 2.2 from \cite{LetacScarsini}, one has
\begin{equation*}
\vect{B}(n) = \vect{B}(0) X(n,1), \quad n=0,1, \ldots
\end{equation*}

Now suppose $X(1)$ satisfies $\ca$ and $\cb$. By Theorem \ref{ext2}, $\law(X(1)) \in \D_d$, with $\widehat{X}^{(1)} = (Y_1, \ldots, Y_d) \sim D_{\vect{t}}$, $\vect{t}$ being the vector from $\cb$. Then clearly $\vect{B}(n) \aas \Lim$ as $n \rightarrow \infty$ for any $\vect{B}(0)$. In particular, for any $k=1, \ldots, d$, we can set $\vect{B}(0) := (0, \ldots,0, 1, 0, \ldots, 0)$, where $1$ is at the $k$th place, to obtain that $\vect{p}_k(n) \aas \vect{Y}$ as $n \rightarrow \infty$, where $\vect{Y}$ has barycentric coordinates $(Y_1, \ldots, Y_d)$ with respect to $(\vect{p}_1, \ldots, \vect{p}_d)$, as required.
\end{proof}

\end{subsection}


\begin{subsection}{Service networks with polling}
The well-known story of Buridan's donkey motivated the authors of \cite{Stoyanov} to consider the following model. Let $\vect{p}_1, \ldots, \vect{p}_d$ be $d \ge 2$ points in the plane. Starting at a point $\vect{R}(0)$ inside the convex hull of $(\vect{p}_1, \ldots, \vect{p}_d)$ at time $0$, at time $td+r$, $t\ge 0$, $r \in \{1,\ldots, d\}$, the donkey moves from the point $\vect{R}(td+r-1)$ to a point $\vect{R}(td+r)$ which is uniformly distributed on the straight line segment connecting the points $\vect{R}(td+r-1)$ and $\vect{p}_r$. As noted in \cite{Letac}, although it is easy to prove existence of the stationary distributions for the $d$ homogeneous Markov chains $\{\vect{R}(td+r)\}_{t \ge 0}$, $r=1, \ldots d$, their computation for $d > 2$ is a difficult problem (the authors of \cite{Stoyanov} focus on the case where $d=2$).

A modification of that scheme in which the donkey travels inside a simplex $S$ with vertices $\vect{p}_1, \ldots, \vect{p}_d$ in $\reals^{d-1}$ by choosing on step $td+r$ its new position at random inside the convex hull spanned by its current location and vectors from the set $\{\vect{p}_1, \ldots, \vect{p}_d\} \backslash \{\vect{p}_r\}$ was considered in \cite{Letac}. The main result of \cite{Letac} establishing the form of the stationary distribution of the thus modified donkey walk was actually proved for yet another version of the model. Namely, rather than the current position of the donkey determining the range for the next step, one instead takes a point with barycentric coordinates (with respect to the affine frame ($\vect{p}_1, \ldots, \vect{p}_d$)) taken from rows of i.i.d.\ random stochastic matrices.

Specifically, let $X := (X_{i,j})$ be a random element of $\SMat_d$, and $\{X(t)\}_{t \ge 0}$ be a sequence of i.i.d.\ random matrices with distribution $\law(X)$. The donkey's position at time $n$ is specified by the vector $\vect{B}(n) = (B_1(n), \ldots, B_d(n)) \in \Simp_d$ of its barycentric coordinates with respect to the affine frame $(\vect{p}_1, \ldots, \vect{p}_d)$. Starting at a non-random point $\vect{B}(0)$, for $r =1, \ldots, d$ and $t=0,1, \ldots,$ given $\vect{B}(td+r-1) = (x_1, \ldots, x_d)$, one has
\begin{equation}\label{Btdr_e1}
\vect{B}(td + r) = (x_1, \ldots, x_{r-1}, 0, x_{r+1}, \ldots, x_d)  +  (x_rX_{r,1}(t), \ldots, x_rX_{r,d}(t)).  
\end{equation}
In other words, setting $\tilde{\vect{p}}_r := \sum_{j=1}^d X_{r,j}(t) \vect{p}_j$, the location of the donkey at time $td+r$ will have the same barycentric coordinates in the frame $(\vect{p}_1, \ldots, \vect{p}_{r-1}, \tilde{\vect{p}}_r, \vect{p}_{r+1}, \vect{p}_d)$ as those those for the donkey's location at time $td+r-1$ in the original frame $(\vect{p}_1, \ldots, \vect{p}_d)$. It is clear that each of the sequences $\{\vect{B}(td+r)\}_{t \ge 0}$, $r=1,\ldots,d$, forms a homogeneous Markov chain. 

This model actually describes the progression of a customer in the following discrete-time closed network with polling. Suppose there are $d$ buffer nodes in the network, accessed in cyclic order by a single server, with customers never leaving the network. At time $n=1$, the server accesses node $1$ and moves customers from that node to other nodes in the network randomly, with transition probabilities taken from the first row of a random $X(1) \in \SMat_d$. At time $n \le d$, the server accesses node $n$, and similarly moves customers from that node to other nodes with transition probabilities taken from the $n$th row of $X(1)$. The same procedure is repeated in a cyclical fashion, with the server accessing node $1$ at time $n = d+1$, and so on, with transition probabilities taken from the rows of random elements $X(j) \in \SMat_d$ for the $j$th cycle. In this formulation, the vector $\vect{B}(n)$ represents the distribution of customers in the network after $n$ steps, with initial distribution $\vect{B}(0)$. 

Markov chains with the cyclic property that each of the sequences $\{\vect{B}(td+r)\}_{t \ge 0}$, $r=1,\ldots,d$, forms a homogeneous Markov chain, were also considered in \cite{Dufresne1998}, where several examples of such Markov chains were provided.

The following assertion was proved in \cite{Letac} (instead of the stated Theorem 3 concerning the ``true donkey walk" discussed at the beginning of this section).

\begin{theorem}\label{t_donkey1}
\emph{(\cite{Letac})}
Let $A = (\alpha_{i,j}) \in \Mat^+_d$ satisfy \eqref{condA}, and $X \sim D_A$. Then, for $r=1, \ldots, d$, the stationary distribution of the $r$th homogeneous Markov chain $\{\vect{B}(td + r)\}_{t \ge 0}$ given by \eqref{Btdr_e1} is Dirichlet with parameters $(\beta_{r,1}, \ldots, \beta_{r,d})$ defined by
\begin{equation}\label{e_vrj}
\beta_{r,j} := \left\{ \begin{array}{ll} 
\sum_{i=j}^r \alpha_{i,j} & \textrm{for $j \le r$},\\ 
\sum_{i=j}^{d+r} \alpha_{i, j} & \textrm{for $j > r$},
\end{array} \right.
\end{equation}
with the convention that $\alpha_{i, j} = \alpha_{i, j'}$ when $j = j'$ \emph{(mod }$d$\emph{)}.
\end{theorem}

For a matrix $X = (X_{i,j}) \in \Mat_d$ and $r=1,2, \ldots, d$, set
\begin{equation*}
\bigr(T_r(X) \bigr)_{ij} := (\delta_{ij} + X_{r,j} \delta_{ri} - \delta_{rj} \delta_{ij}), \quad 1 \le i,j \le d,
\end{equation*}
$\delta_{ij}$ being Kroneker's delta (i.e., $T_{r}(X)$ is the $d \times d$ identity matrix with the $r$th row replaced by the $r$th row of $X$). 

An extension of Theorem \ref{t_donkey1} is given by the following theorem.

\begin{theorem}\label{t_donkey2}
Let $X$ be a random element of $\SMat_d$ such that $T_1(X) \cdots T_d(X)$ satisfies $\cb$. Let $V = (V_{i,j}) \sim G_B$ be a random matrix independent of $X$, where $B \in \Mat^+_d$ has entries $\beta_{i,j}$ given by \eqref{e_vrj}. Furthermore, suppose that
\begin{equation}\label{e_VSX}
V^{(r)} \ed V^{(r-1)} T_{r}(X), \quad 1 \le r \le d,
\end{equation}
where $V^{(0)} := V^{(d)}$.

Then, for $r=1, \ldots, d$, the stationary distribution of the $r$th homogeneous Markov chain $\{\vect{B}(td+r)\}_{t \ge 0}$ is Dirichlet with parameters $(\beta_{r,1}, \ldots, \beta_{r,d})$.
\end{theorem}

\begin{proof}
Let $\{X(t)\}_{t \ge 0}$ be i.i.d.\ with law $\law(X)$, and fix $r \in \{1, \ldots, d\}$. Then
\begin{equation*}
R(t) := \left\{ \begin{array}{ll} 
T_1(X(t)) \cdots T_d(X(t)) & \textrm{for $r = d$},\\ 
T_{r+1}(X(t)) \cdots T_d(X(t)) T_1(X(t)) \cdots T_r(X(t)) & \textrm{for $r < d$},
\end{array} \right.
\end{equation*}
$t=0,1,\ldots$, are i.i.d.\ random elements in $\SMat_d$, and
\begin{equation*}
\vect{B}(td + r) = \vect{B}(r)R(1,t),
\end{equation*}
where $R(1,t)$ is given by \eqref{e_rprod}. Clearly, $R(1)$ satisfies $\cb$ since $T_1(X) \cdots T_d(X)$ satisfies $\cb$, and so by Proposition \ref{prop}, there exists a random vector $\vect{W} \in \Simp_d$ such that
\begin{equation*}
R(t,1) \aas \vect{e}_d^T \vect{W}
\end{equation*}
as $t \rightarrow \infty$. Using Proposition $1$ in \cite{ChamLetac2}, applied to the space $\Simp_d$ and to the random mappings from $\Simp_d$ to $\Simp_d$ defined by $\vect{x} \mapsto \vect{x}Z(t)$, $t \ge 0$, it follows that $\{\vect{B}(td + r)\}_{t \ge 0}$ has a unique stationary distribution.

Now $V^{(r)} \ed V^{(r)}R(1)$ by \eqref{e_VSX}, and therefore Theorem \ref{ext2} implies that $\widehat{R}^{(1)} \sim D_{(\beta_{r,1}, \ldots, \beta_{r,d})}$. Hence the stationary distribution of $\{\vect{B}(td + r)\}_{t \ge 0}$ is $D_{(\beta_{r,1}, \ldots, \beta_{r,d})}$ as required.
\end{proof}

The following remark shows that our Theorem~\ref{t_donkey2} is indeed an extension of Theorem~\ref{t_donkey1}.
 
\begin{remark}
Let $A = (\alpha_{i,j}) \in \Mat_d$ be such that $\alpha_{i \dt} = \alpha_{\dt i} >0$ for $i=1, \ldots, d$, and $B = (\beta_{i,j}) \in \Mat^+_d$, where the numbers $\beta_{i,j}$ are given by \eqref{e_vrj}. Let $(V,Z)  \sim G_B \otimes G_A$, and $X \sim D_A$ defined by \eqref{ZrepX} be such that $T_1(X^{(1)}) \cdots T_d(X^{(d)})$ satisfies $\cb$. Then, for $r=1, \ldots, d$ and using the convention that $r-1=d$ for $r=1$, one has
\begin{align*}
V^{(r-1)}T_r(X) &= (V_{r-1,1}, \ldots, V_{r-1,r-1}, 0, V_{r-1,r+1}, \ldots,  V_{r-1,d}) + V_{r-1,r} \left(\frac{Z_{r,1}}{Z_{r \dt}}, \ldots, \frac{Z_{r,d}}{Z_{r \dt}} \right)\notag\\
&\ed(V_{r-1,1}, \ldots, V_{r-1,r-1}, 0, V_{r-1,r+1}, \ldots,  V_{r-1,d}) + (Z_{r,1}, \ldots, Z_{r,d}),
\end{align*}
where the equality in distribution holds by \eqref{gamprop2} and the fact that $V_{r-1,r} \ed Z_{r \dt}$, as both follow $\Gamma_{\beta_{r-1,r}} = \Gamma_{\alpha_{r \dt}}$. It remains to observe that $\beta_{r,r} = \alpha_{r,r}$ and, for $j \neq r$, $\beta_{r,j} = \beta_{r-1,j} + \alpha_{r,j}$, hence the vector in the last line above is distributed as $V^{(r)}$.

The conditions of Theorem \ref{t_donkey2} are met, and we conclude that, for all $r=1, \ldots, d$, the homogeneous Markov chain $\{\vect{B}(td+r)\}_{t \ge 0}$ has stationary distribution $D_{(\beta_{r,1}, \ldots, \beta_{r,d})}$. In particular, for $X$ defined by \eqref{e_example1}, one has  $\beta_{r,j} = 1$ for $j \neq r+1$ (mod $d$), and $\beta_{r,j} = 2$ for $j = r+1$ (mod $d$). 
\end{remark}

\vspace{0.3cm}
\noindent \textbf{Acknowledgements.} This research was supported by the ARC Centre of Excellence for Mathematics and Statistics of Complex Systems (MASCOS). The author is grateful for numerous fruitful discussions with K. Borovkov, whose suggestions helped in many ways to improve the paper.

\end{subsection}
 

\end{document}